\documentclass[11pt]{amsart}
 \usepackage{amsmath,amssymb}
 \usepackage{xcolor}

\usepackage{geometry}
\newtheorem{theorem}{Theorem}[section]
\newtheorem{thm}[theorem]{Theorem}

\newtheorem{lemma}[theorem]{Lemma}

\newtheorem{corollary}[theorem]{Corollary}

\theoremstyle{definition}

\newtheorem{remark}[theorem]{Remark}

 \theoremstyle{plain}
\newtheorem*{namedthm}{\namedthmname}
\newcounter{namedthm}

\makeatletter

 \newcommand{\R}{\mathbb R}
 
 \newcommand{\C}{\mathbb C}

 \newcommand \PSH {{\rm PSH}}

 \newcommand \Amp {{\rm Amp}}

   \newcommand \Jac{ \rm Jac}
  \newcommand \exph{ \rm exph}

 \numberwithin{equation}{section}

\subjclass[2010]{32W20, 32U05, 32Q15}

\keywords{Envelopes, K\"ahler manifolds, big cohomology classes}

\begin{document}
 \title{The regularity of envelopes}
 \author{Eleonora Di Nezza, Stefano Trapani}
  \maketitle
  \begin{center}
      \emph{In memory of Jean-Pierre Demailly.}
  \end{center}

\maketitle

\begin{abstract}
\noindent Let $X$ be a compact complex manifold of complex dimension $n$ and $\alpha$ be a smooth closed real form on $X$ such that its cohomology class $\{ \alpha \}\in H^{1,1}(X, \R)$ is big. In this paper we prove that, given a bounded  function $f$ with bounded distributional laplacian in $X,$  the $\alpha$-psh envelope $P(f)$ is also locally bounded with  locally bounded distributional laplacian on the ample locus of $\{\alpha\}.$ 
\end{abstract}

\renewcommand{\abstractname}{R\'esum\'e}

\begin{abstract}

Soit $X$ une vari\'et\'e complexe compacte de dimension complexe $n$ et $\alpha$ une (1,1)-forme r\'eelle et ferm\'ee sur $X$ telle que sa classe de cohomologie $\{ \alpha \}\in H^{1,1}(X, \R)$ est grosse. Dans ce travail on d\'emontre que, étant donn\'ee une fonction $f$ dont le laplacien, au sens des distributions, est born\'e, alors l'enveloppe $\alpha$-psh, $P(f)$, a aussi un laplacian born\'e sur le lieu ample de $\{\alpha\}$.

\end{abstract}

\section{Introduction}

\emph{Envelopes of plurisubharmonic functions} played an important role in the development of the pluripotential theory on domains of $\mathbb{C}^n$, see for example \cite{BT,BT82,BT86,S77,Z76}. 

When, relying on the Bedford and Taylor theory in the local case, the foundations of a pluripotential theory on compact  K\"ahler manifolds have been developed \cite{GZ05, GZ07}, \emph{envelopes of quasi-plurisubhar\-monic functions} started to be intensively studied.

The geometric motivations we can mention are, among others, the study of geo\-desics in the space of K\"ahler metrics \cite{Chen00, Dar17, Ber17, RWN17, CTW18, DDL1, CMc19} and the transcendental holomorphic Morse inequalities on projective manifolds \cite{WN19}.

The two basic (and related) questions are about the regularity of envelopes and the behaviour of their Monge-Amp\`ere measures. 
To fix the setting, let $X$ be a compact complex manifold of complex dimension $n$, let $\alpha$ be a smooth closed real $(1,1)$-form and let $f$ be a function on $X$ bounded from above. We are going to refer to $f$ as ``barrier function". Then the prototype of an envelope construction is 
$$ P_\alpha(f):= \left(\sup\{ \varphi\in \PSH(X, \alpha)\,: \; \varphi\leq f \}\right)^*,$$ 
where $^*$ denotes the upper semi-continuous regularization and $\PSH(X, \alpha)$ is the space of all $\alpha$-plurisubharminic functions, as defined in Section 2.\\
The function $P_\alpha(f)$ is either a genuine $\alpha$-plurisubharmonic function or identically $-\infty$. 
When $ f= -{\bf 1}_T$ is the negative characteristic function of a subset $T$, then $P_\alpha (f) = f^*_T$ is the so-called relative extremal function of $T$ \cite{GZ05}.
When $f = 0$, then $P_\alpha(0)= V_\alpha$ is a distinguished potential with minimal singularities.\\ 

We recall that a function on some open set in $X$ is said to be in $C^{1,\bar{1}}$
if it is locally $L^{\infty}$ and its distributional Laplacian is represented by a locally bounded function. Any $C^{1,\bar{1}}$ function is also $ C^{1,\alpha}$ for any $0 < \alpha <1.$ In this paper we prove the following:
\begin{thm}
Let $\alpha$ be a real closed $(1,1)$-form such that $\{\alpha\}$ is a big class. Assume $f\in C^{1, \bar{1}} (X)$. Then $P_\alpha(f)$ is $C^{1, \bar{1}}$ on the ample locus of $\{\alpha\}$.
\end{thm}
We actually prove a regularity result for the more general rooftop envelope $P_\alpha (f_1, \cdots, f_k)$ (see Theorem \ref{main}).

The study of such envelopes has led to several works. We start  by summarizing them in the case of a smooth barrier function $f$.\\
The first result to mention is \cite{Ber09}, where the author proves that in the case $\alpha\in c_1(L)$ where $L$ is a big line bundle over $X$, the envelope $P_\alpha(f)$ is  ${C}^{1,1}$ on the ample locus $\Amp(\{\alpha\})$ of $\alpha$, and moreover 

\begin{equation}\label{support}
\alpha_{P_\alpha(f)}^n= {\bf 1}_{\{P_\alpha(f)=f\}} \alpha_f^n.
\end{equation}
Here, given a $\alpha$-psh function $u$, $\alpha_u:=\alpha + i \partial \bar{\partial}u$ and $\alpha_{u}^n$ denotes the non-pluripolar Monge-Amp\`ere measure
$(\alpha + i \partial \bar{\partial}u)^n$ (see \cite[Definition 1.1]{BEGZ10}). Thus the left-hand side of \eqref{support} is the non-pluripolar product of $P_\alpha(f)$ while $\alpha_f^n$ is the wedge product $n$ times of the form $\alpha_f$ that has bounded coefficients since $f$ is smooth, (observe that $f$ is not $\alpha$-plurisubharmonic so the non-pluripolar product of ${\alpha_f}$ does not make sense).

Later people started to work on possible generalisations of the above results in the case of a pseudoeffective class $\{\alpha\}$ that does not necessarily represent the first Chern class of a line bundle. Assuming that $\{\alpha\}$ is big and nef, Berman \cite{Ber19}, using PDE methods, proved that the envelope $P_\alpha(f)$ is $C^{1,\bar{1}}$ on $\Amp(\{\alpha\})$ and that the identity in \eqref{support} holds. \\

The optimal regularity $C^{1,1}$ when $\{\alpha\}$ is a  K\"ahler class was then proved independently by \cite{T18} and \cite{CZ19}, while the big and nef case was settled in \cite{CTW18}.\\
For a general pseudoeffective class the equality 
\begin{equation}\label{support1}
\alpha_{P_\alpha(f)}^n = {\bf 1}_{\{P_\alpha(f)=f\}} \alpha_f^n.
\end{equation}
is established in \cite{DNT}, without relying on the regularity of $P_\alpha(f)$.

The general case of a big class was treated for the first time by Berman and Demailly. In \cite{BD} the authors claim to prove the $C^{1,\bar{1}}$ regularity of $P_\alpha(f)$ on the ample locus of $\{\alpha\}$ when $f$ is smooth. However it became clear later that their arguments had a mistake in their crucial technical Lemma 1.12.
In fact, the lower bound in their Lemma 1.12 does not follow from \cite[eq. (1.8)]{BD} (as they state) since there are some mixed terms to take care of.  

The purpose of this paper is then to correct the proof in \cite{BD} and generalise their regularity result.

 Our proof relies on results in \cite{BD} and especially on \cite{Dem94} but it presents two novelties. We proceed in three steps:
\begin{itemize}
    \item As in \cite{BD}, we consider the regularization of the envelope through the exponential map. We first derive weaker estimates than the one stated in \cite[Lemma 1.12]{BD}, which allows us to obtain H\"older regularity of the envelope. This is done in Lemma \ref{ineq00}.
    \item Such regularity and a version of the Lelong-Jensen inequality, allow us to prove an estimate for the complex hessian of the regularized function (like the one in Lemma 1.12) but with variable coefficients (Lemma \ref{lemma_varcoeff}).  We can however control the asymptotics of these coefficients near the boundary of the the ample locus of $\{\alpha\}$;
    \item In the third step, we pass to a modification, and we transport the regularized function on a suitable line bundle. Thanks to a change of coordinates we are able to ``kill" the coefficients appearing in the lower bound of the hessian, thus reducing to estimates with constant coefficients. This is the key to use (more or less) the same strategy as in \cite{BD} and complete the proof. This is done in Section \ref{section_line}.
\end{itemize}

\medskip\noindent\textbf{Acknowledgements.}  
The second author is partially supported by PRIN \emph{Real and Complex Ma\-ni\-folds: Topology, Geometry and holomorphic dynamics} n.2017JZ2SW5, and by MIUR Excellence Department Projects awarded to the Department of Mathematics, University of Rome Tor Vergata,2018-2022 CUP E83C18000100006, and 2023-2027. The paper was finalized while the second author was visiting C.M.L.S supported by the CNRS grant PEPS 2021 of the first author. 
 We also thank Vincent Guedj, Ahmed Zeriahi and S\'ebastien Boucksom for very useful discussions on the subject of this paper. 

\section{Preliminaries}\label{s1}
Let us first assume that $(X,\omega)$ is a compact K\"ahler manifold of complex dimension $n\geq 1$. The K\"ahler assumption will be then removed at the end of Section \ref{section_line}.\\
Let $\alpha$ be a smooth closed real $(1,1)$ form such that the cohomology class $ \{ \alpha \}$ is big.   

The cohomology class $ \{\alpha\} $ is said to be big if there exists a positive closed $(1,1)$-current $T = \alpha + i \partial \bar{\partial} \Phi \in \{\alpha\}$ such that $T \geq  \varepsilon_0\omega$, for some $\varepsilon_0>0$. By replacing $\omega$ with $\varepsilon_0 \omega$ we can assume 
$T \geq \omega.$
Such currents are called K\"ahler currents.

Let $\Omega $ be the set of $x$ in $X$ such that  there exists a K\"ahler current cohomologous to $\alpha$ which is smooth in a neighborhood of $x.$ This is called the ample locus of $\{ \alpha \}$ (also denoted by $\Amp{\{\alpha\}}$ in the literature). If $\{\alpha\}$ is big, $\Omega$ is a non-empty open subset  such that $\Sigma= X \setminus \Omega$ is a closed analytic set. Moreover, there exists a  K\"ahler current 
$T = \alpha + i \partial \bar{\partial} \Phi$ with analytic singularities which is smooth in $\Omega$  (see \cite{Bou04}). Observe that, without loss of generality, we can normalize $\Phi\leq -1$.

The cohomology class of $\alpha$ is K\"ahler if and only if $\Omega = X.$ Also, $\Phi$ is a smooth function and we can take $\omega=\alpha + i \partial \bar{\partial} \Phi$.

 A function $\varphi: X \rightarrow \mathbb{R}\cup \{-\infty\}$ is called quasi-plurisubharmonic if locally $\varphi= \rho + u$, where $\rho$ is smooth and $u$ is a plurisubharmonic function. We say that $\varphi$ is $\alpha$-plurisubharmonic ($\alpha$-psh for short) if it is quasi-plurisubharmonic and $\alpha+i \partial \bar{\partial}\varphi \geq 0$ in the weak sense of currents on $X$. We let $\PSH(X,\alpha)$ denote the space of all $\alpha$-psh functions on $X$. \\
 Let $f$ be a smooth function on $X,$  we denote by $P(f)$ the $\alpha$-psh envelope of $f,$ given by 
$$ P(f) = \left(\sup\{ \varphi \;\,  \alpha\mbox{-psh} \ :  \ \varphi \leq f \}\right)^*.$$
Note that we drop the index $\alpha$ for simplicity since no ambiguity can occur.\\
I can be showed that $P(f)$ is $\alpha$-psh as well, is bounded by $f$ and hence there is no need to take the upper semicontinuous regularization. Also, clearly $P(f) - f$ is the largest non-positive   $(\alpha + i \partial \bar{\partial} f)$-psh function, so for simplicity we will assume in the sequel $f \equiv 0.$ Observe that by definition $\Phi \leq P(0),$ hence $P(0)$ is locally bounded in $\Omega.$

\smallskip

Let $x_0 \in \Omega$ and let $z$ be holomorphic local coordinates with center in $x_0$ which are $\omega$ normal at $x_0.$ Given a function in $\Omega,$ we will make the harmless confusion between the function in $\Omega$ restricted to a small coordinate open set (and in this case we use the variable $x$), and the same function expressed in the local coordinate variable $z.$ 

Let $\varphi$ be any $\alpha $-psh function, we denote by $\Psi(\varphi)$ the regularised function of $\varphi$ defined in \cite[eq. (1.6)]{BD} as
$$ \Psi(\varphi)	(z,w) = \int_{ \zeta \in T_z X } \varphi({{\exph}_{z} }(w \zeta)) \, \chi (|\zeta|^2) \, dV_\omega, \quad (z,w) \in X\times \mathbb{C} $$
where $\exph : TX \rightarrow X$, $T_zX \ni \zeta \rightarrow  {\exph}_z(\zeta)$ is the formal holomorphic part
of the Taylor expansion of the exponential map 
of the Chern connection on $TX$
associated with the metric $\omega$ and $\chi : \mathbb{R}\rightarrow  \mathbb{R }^+ $ is a smooth cut-off function  such that $\chi > 0 $ for $t < 1,$ $\chi(t) = 0$ for $t \geq 1$ and $\int_{\mathbb{C}^n} \chi(|\zeta|^2) d \lambda(\zeta) = 1.$

We simply denote by $\Psi$ the function $\Psi(P(0)).$ Observe that $\Psi \leq 0.$

 By \cite{Dem94} there exists $K > 0 $ large enough such that  for $w \in \mathbb{C} \setminus \{0 \}$  we have 
\begin{equation} \label{Lelong}
\lambda(z,w) := \frac{\partial ( \Psi(z,e^s) + Ke^{2s} ) }{\partial s}_{\Big|s = \log|w|} \rightarrow \nu_z(P(0))   
\end{equation} 
as $w$ goes to $0$ where $\nu_z(P(0))$ is the Lelong number of $P(0)$ at the point $z.$ Moreover,  the function $\lambda$ is continuous, non-negative and increasing with respect to $|w|$ (see \cite[page 15]{Dem94}.
In particular the function $$ \Psi(z,e^s) + Ke^{2s} $$ is increasing with respect to $s.$

\section{Lower bound for the hessian with variable coefficients}
Let $\theta = \sum_j \theta_j \frac{\partial}{\partial z_j}$ and $\tau = \tau_1  \frac{\partial}{\partial w}$ be tangent vectors in the coordinates $z$ and $w$ and let $S$ be a real $(1,1)$-form in the coordinates $z,$ $w$ and $(z,w)$ (i.e. $S$ can be locally written as $\sum f_{jk} dz_j\wedge d\bar{z}_k + g dw\wedge d\bar{w} + \sum( h_i dz_i\wedge d\bar{w}+\tilde{h}_i dw \wedge d\bar{z}_i)  $).

As in \cite{Dem94} we denote by $S[\theta]^2, \  S[\tau]^2, \  S[\theta,\tau]^2$  the action of the form $S$ on the vectors $\theta$, $\tau,$ and   $\theta + \tau$ respectively    as an hermitian form.
For notational convenience we denote by  $|dz|^2, |dw|^2$  the  forms 
$$ |dz|^2 :=  (i/2) \sum_j d z_j \wedge d \bar{z_j},   \  |dw|^2 = (i/2) dw \wedge d \bar{w} $$ 
Observe that if the holomorphic coordinates are normal at $0,$ then
$|dz|^2_0 = \omega|_0.$
Note also  that $|dz|^2[\theta]^2 = \sum_j |\theta_j|^2 := | \theta|^2,$ and $|dw|^2[\tau]^2 =  |\tau_1|^2 :=  |\tau|^2.$ 
Moreover if $p,q$ are smooth  non-negative functions in $z,w$ coordinates, then:  
\begin{equation}\label{square} 
\left(\frac{ p^2 |dz|^2}{2} + \frac{q^2 |dw|^2}{2 }\right)[\theta,\tau]^2  -pq |\theta||\tau| = \frac{1}{2}( p |\theta| - q|\tau|)^2 \geq 0.
\end{equation}  

\medskip

Given $\beta \in [0,1]$ we denote by $I_{\beta}$ the following inequality:
\begin{flalign*}
&\left( \alpha(z)+ i \partial \bar{\partial}_{(z,w)} \Psi(z,w) \right)[\theta,\tau]^2 \\
&\geq 
 -\left(A\lambda(z,|w|) +K|w|^{1 + \beta} \right) |\theta|^2 - K|w|^\beta |\theta||\tau|  -K|w|^{\beta}  \ |\tau|^2
\\
&= \left( -\left(A\lambda(z,|w|) +K|w|^{1 + \beta} \right) |dz|^2   -K|w|^{\beta}  \ |dw|^2 \right)[\theta,\tau]^2  - K|w|^\beta |\theta||\tau|  
\end{flalign*} 

where  $ A > 0, K > 2$ and $0 < |w| < \delta_0.$

\smallskip

\begin{lemma}\label{ineq00}
The inequality $I_0$ holds for some positive constants $A,K,\delta_0.$  \label{firstineq} \end{lemma}
We thank the referee for pointing out that the above lemma was already proved in \cite[Lemma 4.1]{KN19}. However we prefer to give a proof since we make considerations which will be crucial for us in a more general case in the sequel.

\begin{proof}
We choose as above canonical local holomorphic coordinates $(z,\zeta)$ on the tangent bundle of $X$ centered at some point $(x_0,0)$ (which is then identified with $\C^n$), where $z$ varies in an open set $U,$ and the coordinates  are $\omega$ normal at $x_0.$ Let $z\in U$ and $B(z,r) := \{ z' :   |z'-z| < r \}.$
Let $\eta$ be a local potential of $\alpha,$ then
$P(0) + \eta$ is psh. 
\noindent Observe that, since $\Psi$ is a linear operator and $\eta $ is a smooth function
\begin{equation}\label{eq00}
\Psi(P(0) + \eta) = \Psi(P(0)) + \Psi(\eta) = \Psi(P(0)) + \eta + O(|w|^2).
\end{equation}

We let $A=\sup_{|\zeta|\leq 1,|\xi|\leq 1} \{|c_{jklm}\zeta_j \bar{\zeta}_k \xi_l \bar{\xi}_m\}|$, where $c_{jklm}$ are the coefficients of the curvature tensor of $\omega$. 

The above estimate comes from the proof of \cite[Theorem 4.1]{Dem94} where we choose $u = A \omega$ and $\gamma= 0.$ Indeed, it follows from \cite[Proposition 3.8]{Dem94} (applied with $N=2$), the estimates (3.11) and (4.3) in \cite{Dem94} that for $K' >>1$ large enough we have 
the following estimate up to a term $O(|w|^2)$:
$$
\alpha +i \partial \bar{\partial}_{(z,w)} \Psi(z,w)[\theta,\tau]^2 \geq  
 $$
$$ |w|^2\int_{\C^n} -\chi_1(\zeta) \sum_{j,k,l,m} \frac{\partial^2 (P(0)+\eta)}{\partial z_m \partial \bar{z_l}} ({\exph}_{z}(w\zeta)) \left(c_{jklm} +\frac{1}{|w|^2} \delta_{jm}. \delta_{kl}\right) \tau'_j \bar{\tau'_k}\; d\lambda(\zeta) - K'(|\theta||\tau|+|\tau|^2)  $$
where $\chi_1$ denotes the primitive $\chi_1(t) =\int_{+\infty}^t \chi(u) \, du $ of $\chi$ such that $\chi_1(t) = 0$ for $t\geq  1$ and ${\tau'}=\theta+\zeta \tau +O(|w|)$. \\
Since $0\leq -\chi_1$, we get (always up to a term of the form $O(|w|)$)
\begin{eqnarray*} 
 &&\alpha+i\partial \bar{\partial}_{(z,w)} \Psi(z,w)[\theta,\tau]^2 \\
&&\geq  -A \left( \int_{\C^n} -|w|^2 \chi_1(\zeta) \sum_{j,k,l,m} \frac{\partial^2 (P(0)+\eta)}{\partial z_m \partial \bar{z_l}} ({\exph}_{z}(w\zeta)\; d\lambda(\zeta) \right) |\theta|^2 - K''(|\theta||\tau|+|\tau|^2)\\
&&\geq -A\lambda_\Omega(z,w)  |\theta|^2 - K''(|\theta||\tau|+|\tau|^2)
\end{eqnarray*}
 where the function $\lambda_{\Omega}$ is defined in \cite[page 14]{Dem94}. We stress that the coefficients $K', K''$ are $O(1)$ thanks to \cite[eq. (3.11)]{Dem94}.
\smallskip

Now the various ``$O$-terms"  depend on the Taylor expansion of ${\exph}_{z}$ up to order $N = 2,$ in the given local normal holomorphic coordinates In particular they depend on the coefficients $b_{jklm}$ and $d_{\alpha km}$ appearing in \cite[Proposition 2.9]{Dem94}. On the other hand, thanks to \cite[Main Theorem]{JG} there exist holomorphic normal coordinates centered at a point $x_0$  such that, in these coordinates, 
the coefficients of the Taylor expansion of the  exponential map at $0$ depend  only  on the  curvature of $\omega$ and its subsequents covariant derivatives at $x_0.$ Since $\omega$  is defined on the compact manifold $X$ it follows that these coefficients are uniformly  bounded.

Also, thanks to \cite[eq. (4.5)]{Dem94} we know that
\begin{equation}\label{lelong_a}
| \lambda - \lambda_{\Omega} |  = O(|w|^2).
\end{equation}
Thus

$$
\left( \alpha(z)+ i \partial \bar{\partial}_{(z,w)} \Psi(z,w)\right)[\theta,\tau]^2 \geq  
 $$
$$  \left( -\left(A\lambda(z,|w|) + \tilde{K}|w|^2 \right) |dz|^2   - \tilde{K} |dw|^2 - C|w|\,( |dz|^2 + |dw|^2) \right)[\theta,\tau]^2  - \tilde{K}|\theta||\tau|  $$
 where $ |w| < \delta_0 \leq 1.$ \\ Let us stress that $-C|w| \,( |dz|^2 + |dw|^2)[\theta, \tau]^2$ is needed to bound the term $O(|w|)[\theta, \tau]^2$ appearing in \cite[Proposition 3.8]{Dem94} applied with $N=2$. 
Since $|w| \leq 1,$  if we put $K = \tilde{K} + C$ we find $I_0.$ \end{proof}

\begin{remark}\label{i5}
In \cite{DHGKZ} the authors study the H\"{o}lder regularity of solutions of Monge-Amp\`ere equation, in particular in their Theorem D they work in the case of a big class $\{\alpha\}$. 

In their proof they make use of the incorrected \cite[Lemma 1.12]{BD}. However we point out that the inequality $I_0$ is enough for their purposes. The interested reader can find all the details in \cite[Theorem 4.2]{KN19} where the authors considered the case when the reference form is an hermitian positive definite form. The same arguments apply when $\alpha$ represents a big class.
\end{remark}

\smallskip

\noindent For simplicity we make the change of coordinate 
$(z,w) \to (z,\delta_0^{-1} w),$ so we will work with $|w|<1$. Let $\gamma \in [1/2,1],$ we say that the statement $II_{\gamma}$  holds if there exist  strictly positive continuous functions  $C_1(x), \delta(x)$ in  $\Omega$ and a positive constant $C_2$  such that 
$$-C_2 \leq   \left( \frac{ \Psi(x,\delta) - P(0)(x)}{\delta^{ 2 \gamma}} \right)  \leq  C_1(x)$$ for $0 < \delta < \delta(x). $ 

In the above inequality, we used the variable $x$ and not the local coordinate $z$ in order to stress that the inequality is global.

\smallskip

The main strategy to prove regularity of the envelope both in \cite{BD} and here, is to prove the inequality $II_1$  with $C_1(x) < C_1$, $C_1 > 0 $ and $\delta(x) > \delta_0 > 0 $ where $x$ varies on  some relatively compact open set of the ample locus of the class $\{ \alpha\}.$ It is proved in \cite[page 44]{BD} that this uniform inequality is indeed sufficient to obtain $C^{1,\bar{1}}$ regularity of $P(0)$ on the ample locus of $\{ \alpha \}$. However, given the central importance of the latter  statement, we will give  details of the proof in the lemma below:

\begin{lemma}\label{ineqreg}
Assume $$-C_2 \leq   \left( \frac{ \Psi(x,\delta) - P(0)(x)}{\delta^{ 2}} \right)  \leq  C_1 $$
for all $x \in U,$ where $U$ is a relatively compact open set of $\Omega,$ and with $C_1>0, C_2>0$, $0 < \delta < \delta_0 > 0$ which may depend on $U$. 
Then $P(0)$ is $C^{1,\bar{1}}(\Omega).$  
\end{lemma}
\begin{proof}
Let $\mu_0$ be a finite and positive Borel measure on a relatively compact  open set $U$ in $\mathbb{R}^{N}.$ Let $\mu_0 = \mu_1 + \mu_2$ be the Lebesgue decomposition of $\mu_0$ with respect to the Lebesgue measure $m.$  Then $\mu_1$ and $\mu_2$ are finite and positive. Also $\mu_1$ is absolutely continuous with respect to $m$ whereas $\mu_2$ is the singular part (with respect to $m$). We define $$\underline{D}(\mu_0)(x) := \liminf_{r \to 0}\frac{ \mu_0(B(x,r))}{m(B(x,r)}$$ where $B(x,r)$ is the open ball of center $x$ and radius $r.$ Similarly we define  $$\underline{D}(\mu_1)(x) := \liminf_{r \to 0}\frac{ \mu_1(B(x,r))}{m(B(x,r)}.$$ These are globally defined functions on $U,$ and they may have non-negative values including $+ \infty.$  \\
 \cite[Theorem 7.13]{Rudin} ensures that  $\underline{D}(\mu_0)(x) = \underline{D}(\mu_1)(x) $ outside of a set of $m$-measure zero (i.e. a set of measure zero w.r.t $m$). Moreover outside of a set of $m$-measure zero $\underline{D}(\mu_0)(x) $ and  $\underline{D}(\mu_1)(x) $ are in fact limits (and not only $\liminf$). For such $x$ we (respectively) denote by $D\mu_0(x)$ and $D\mu_1(x)$ these limits. Moreover it follows from \cite[Theorem 7.8]{Rudin} that
$$ \mu_1(E) = \int_E {D}(\mu_1) dm, $$
for any  $E$ Borel subset of $U$. 
On the other hand
$D(\mu_1) = \underline{D}(\mu_1)= \underline{D}(\mu_0)$  almost everywhere with respect to $m,$ hence
$$\mu_1(E) = \int_E \underline{D}(\mu_0) dm.$$
We claim that if the function $\underline{D}(\mu_0)$ is bounded, then $\mu_0 = \mu_1,$ i.e. $\mu_0$ is absolutely continuous with respect to $m$ with bounded density. In fact 
in \cite[Theorem 7.15]{Rudin} it is proved that the set $ S_2=\{x \in U :  \underline{D}(\mu_2)(x) = + \infty\}$ as full measure with respect to the singular measure $\mu_2,$ meaning that $\mu_2(U)=\mu_2(S_2)$. 
Since
$$ S_2\subseteq S_0= \{ x \in U :  \underline{D}(\mu_0)(x) = + \infty \}  = \{ x \in U : {D}(\mu_0)(x) = + \infty \} $$ it follows that $\mu_2(U)=\mu_2(S_0)$. 
Note also that since $\mu_1$ is  finite and positive, and  $\underline{D}(\mu_0) $ is a density of $\mu_1$ (w.r.t $m$) then the Lebesgue measure of $S_0$ is zero.\\
Since by assumption the function  $ \underline{D}(\mu_0) $ is bounded, the set $S_0$ is empty. Thus $\mu_2(U) = \mu_2(S_0) = \mu_2(\emptyset) = 0.$ Then $ \mu_0 = \mu_1$ is absolutely continuous and it has a bounded $m$-density. This proves the claim.

Now we choose a small open set $U$ in $\Omega$ such that $\alpha = i \partial \bar{\partial} \eta$ in $U.$ The function $\varphi := P(0) + \eta$ is then psh on $U$ and so $\mu_0 := \Delta \varphi$ is a finite and positive measure on $U$. 
If follows from the assumption
$$-C_2\leq \left( \frac{ \Psi(x,\delta) - P(0)(x)}{\delta^{ 2}} \right)\leq C_1$$ and from \cite[eq. 1.16]{BD} that $\underline{D}(\mu_0)$ is bounded on $U.$
It then follows from the arguments above (with $N=2n$) that the measure $\Delta \varphi$ is absolutely continuous with respect to the Lebesgue measure and it has bounded density. This implies that $\varphi,$ hence $P(0)$, is $C^{1,\bar{1}}$ on $U$. Since $U$ is any relatively compact set in $\Omega$ we can then conclude that $P(0)\in C^{1,\bar{1}}(\Omega)$.
\end{proof}

We also have the following: 

\begin{lemma} \label{beta}

Fix $A,K, \beta$ real numbers with $A > 0,$ $  K> 2$, $  \beta \in [0,1].$ Assume the inequality $I_{\beta}$. Then the statement $II_{\gamma}$    holds as well  with $\gamma = \frac{\beta + 1}{2}.$  Moreover if $\{ \alpha \}$ is a K\"ahler class, we can  choose $C_1(x)$ and $\delta(x)$ to be positive constants.

 \end{lemma}
 
\begin{proof}
Note that if $\beta \in [0,1]$ then $\gamma \in [1/2,1].$
Given $ K' ,c$ positive real numbers, we consider the functions 
$$ f_{\delta}(x,w) =   \Psi(x,w)   +K' |w|^{2\gamma} - K' {\delta}^{2\gamma}  + c \log \delta $$ 
and  
$$ \psi_{c,\delta}(x)  = \inf_{ 0 < |w| < \delta } (f_{\delta} - c \log(|w|).$$
for $0 < \delta < 1, x \in X $ and $|w| < \delta.$ 
Fix $x_0 $ in $\Omega$ and let $z$ be the holomorphic local coordinates centered in $x_0$ which are $\omega$ normal and defined in an open neighborhood of $\bar{U}$ where $U$ is an open neighborhood of $x_0.$
Simple computations give 
$$ \alpha +  i \partial \bar{\partial}_{(z,w)}( f_{\delta} - c \log|w|) = \alpha + i \partial \bar{\partial}_{(z,w)} f_{\delta}  =  \alpha + i \partial \bar{\partial}_{(z,w)} \Psi + 2K' \gamma^2 |w|^{2(\gamma-1)}|dw|^2. $$
We then use the inequality (\ref{square}) with $p = K|w|^{\frac{\beta +1}{2}}, q = |w|^{\frac{\beta -1}{2}},$  and thanks to $I_{\beta},$ we find:

 \begin{flalign*}
& \alpha + i \partial \bar{\partial}_{(z,w)} f_{\delta} \\
 &\geq  -\left(A\lambda(z,|w|) + ( K + K^2/2) |w|^{ \beta + 1 } \right) |dz|^2  +  \left( 2K' \gamma^2 |w|^{2(\gamma-1)} -K - \frac{|w|^{\beta-1}}{2}\right)|dw|^2  \\
&\geq   -\left(A\lambda(z,|w|) + K^2 |w|^{ \beta +1} \right) |dz|^2  + |w|^{2(\gamma -1)}\left( 2K' \gamma^2 - K |w|^{2(1 - \gamma)} - \frac{1}{2}|w|^{\beta + 1 - 2 \gamma} \right)|dw|^2,
\end{flalign*}
{where in the first inequality we used the fact that $-K|w|^\beta|dw|^2 \geq -K|dw|^2$.}\\
We require that \begin{equation} 2K' > \frac{2K + 1}{2 \gamma^2}, \label{K'1} \end{equation}
in order to have $2K' \gamma^2 - K - \frac{1}{2} > 0 $. Since $ \gamma = \frac{\beta +1}{2} \leq 1,$ and $ 0 < |w| \leq  \delta$,  we find 
 
 \begin{equation}\label{psh1} \alpha + i \partial \bar{\partial}_{(z,w)} f_{\delta}  >  -\left(A\lambda(z,|w|) + K^2 |w|^{ \beta +1} \right) |dz|^2 . \end{equation} 
By definition of $\lambda$ (see \eqref{Lelong}) and $f_\delta$ we have that 
$$  \frac{\partial  f_{\delta}(z,e^s)  }{\partial s}= \lambda(z,e^{s}) - 2Ke^{2s} + 2 \gamma K' e^{2 \gamma s} = \lambda(z, e^{s}) + 2e^{2 \gamma s} (  \gamma K' - K e^{2(1- \gamma)s}).$$
Now we also require that  
\begin{equation} K' > \frac{ K}{\gamma}    \label{K'2} \end{equation}
and we find that 
\begin{equation}\label{eq lelong}
 \frac{\partial  f_{\delta}(z,e^s)  }{\partial s} > \lambda(z, e^{s})>  0 
 \end{equation}
 for $s \leq  0. $ For example we could take 

$$ K' = 2K + 2. $$
Since the function $P(0)$ is finite at $x_0$ (recall that $x_0\in \Omega$), by (\ref{Lelong}) we have 
\begin{equation} \lim_{s \rightarrow - \infty}  \frac{\partial  f_{\delta}(z,e^s)  }{\partial s} = 0. \label{limit} \end{equation} \\
By \eqref{psh1} and the fact that $|w|\leq 1$, it follows that 
\begin{equation*}
\alpha + i \partial \bar{\partial}_{(z,w)} f_{\delta}  >  -\left(A\lambda(z,|w|) + K^2 \right) |dz|^2.
\end{equation*}
 By  \eqref{limit} the function  $s \rightarrow f_{\delta}(z,e^s) - cs$ is strictly decreasing near $- \infty,$ then this function takes a minimum value  in the interval $- \infty < s \leq \log\delta$ at some point $s_0. $  Set $t_0 = e^{s_0}.$ If $t_0 = \delta$ then $\frac{\partial  f_{\delta}(z,e^s) }{\partial s} - c  \leq 0 $ in $(-\infty, \log \delta]$, in particular  by \eqref{eq lelong} $\lambda(z, t_0) < c.$ If $t_0 < \delta, $ then $ \frac{\partial  f_{\delta}(z,e^s)  }{\partial s}_{\big|s=s_0} = c,$ and again $\lambda(z,t_0) < c.$ 
 Let $\eta$ and  $\sigma$ be   local potentials  of $\alpha$ and  of $\omega$, respectively, in an open neighborhood of $\bar{U}.$ It  follows from inequality \eqref{psh1} that the function  $\eta(z) + C_0 \sigma(z)+ f_{\delta}(z,w) - c \log |w| $ is strictly plurisubharmonic in $z$ and $w$ in an open  neighborhood of the point $(x_0,t_0)$  in $X \times ( \mathbb{C} \setminus \{0 \}),$ of the form $V \times \{ t_0 - \rho < |w| < t_0 + \rho \}$ for some $\rho > 0$ where 
$$ C_0 := Ac + K^2\delta^{2 \gamma} $$ (recall that $2\gamma=\beta+1$). If we apply the above to every point $x_0 \in \Omega,$ by Kiselman's minimum principle and by (\ref{psh1}), we have  that $\alpha + i \partial \bar{\partial} \psi_{c,\delta} \geq - C_0 \omega$  in $\Omega.$ 
Moreover $\psi_{c,\delta}(x) \leq  f_{\delta}(x, \delta) - c\log(\delta) \leq 0$ since $\Psi \leq 0,$  so, since $X \setminus \Omega$ is a closed analytic set,  the estimate $\alpha + i \partial \bar{\partial} \psi_{c,\delta} \geq - C_0 \omega$ extends to all of $X.$\\
 Let $ \mu := \frac{C_0}{C_0 + 1},$ so that $0 < \mu < 1,$ and consider the function $ u =  (1- \mu) \psi_{c,\delta}(z) + \mu \Phi(z),$ then we have $u \leq 0$ and 
$\alpha + i \partial \bar{\partial} u \geq   ( -(1 - \mu)C_0  + \mu)) \omega = 0.$ 
So, the function $u$ extends to an $\alpha$-psh non-positive function on $X,$ in particular $u \leq P(0).$
If $x_0 \in \Omega$, then in local $\omega$ normal  holomorphic coordinates this gives:
$$ P(0)(z) \geq (1 - \mu) \psi_{c,\delta}(z) + \mu \Phi(z) \geq    \psi_{c,\delta}(z) + \mu \Phi(z) $$  since   $\psi_{c,\delta}(z) \leq 0.$  In other words:  

$$ P(0)(z) \geq  \Psi(z,t_0) + K' t_0^{2 \gamma} - K' \delta^{2 \gamma} -c \log(t_0/ \delta) + \mu \Phi(z).  $$ 
On the other hand by \eqref{eq lelong}, the function $t \rightarrow \Psi(z,t) + K' t^{2 \gamma}$ is increasing and it coincides with $P(0)(z)$ at $t=0,$ so 
\begin{equation}  P(0)(z) \leq \Psi(z,t_0) + K' t_0^{2 \gamma} \leq P(0)(z)  + K' \delta^{2 \gamma}  + c \log(t_0/ \delta) - \mu \Phi(z).  \label{mainineq1} \end{equation}     
From (\ref{mainineq1}) we first derive
$$ c \log(t_0/ \delta)  \geq - (  K' \delta^{2 \gamma} + \mu |\Phi(z)|)$$ (recall that $\Phi \leq -1$), which  is equivalent to 
\begin{equation}  \delta \exp \left( -\frac{1}{c}( K' \delta^{2 \gamma} + \mu |\Phi(z)|) \right)  \leq t_0  \leq \delta. \label{Phi} \end{equation}   
Again from (\ref{mainineq1}) we derive:
\begin{equation*}  0 \leq   \Psi(z,t_0)  - P(0)(z) +  K' t_0^{2 \gamma} \leq    K' \delta^{2 \gamma}  + c \log(t_0/ \delta) + \mu  |\Phi(z)| \leq 
K' \delta^{2 \gamma}   + \mu  |\Phi(z)|. \label{ineq0} 
\end{equation*} 
Observe that   $  \frac{\mu}{C_0} \leq 1, $ and if we let $c = \delta^{2 \gamma},$  then $ \frac{C_0}{c} =  A + K^2$,   then $\frac{\mu}{c}  \leq  A + K^2 $ and $K' + \frac{\mu}{c} | \Phi(z)| \leq   K' +  \left(A + K^2\right)|\Phi(z)|  \leq B |\Phi(z)|$ where $B>1$ is constant. We set
\begin{equation} R(x) := B|\Phi(x)|. \label{R} \end{equation}
Therefore 
 \begin{equation} \delta \exp(-R(z)) \leq t_0 \leq \delta. \label{ineq1} \end{equation}
 and
\begin{equation} 0 \leq  \frac{ \Psi(z,t_0)  - P(0)(z) +  K' t_0^{2 \gamma}}{ \delta^{2 \gamma} }  \leq  R(z).  \label{ineq2} \end{equation} Observe that the functions $\Psi$ and $R$ are defined in all of $\Omega$ not only locally. 
For fixed $x \in \Omega$  and $0 < \delta' < \exp(-R(x))$ we choose $ \delta := \exp(R(x)) \delta'   < 1.  $ Observe that by \eqref{ineq1} we have $\delta'\leq t_0$. Since $\Psi(x,t) + K' t^{2 \gamma}$ is increasing in $t$ we then get
\begin{eqnarray}\label{ineq_final_gamma}
\nonumber 0 &\leq &  \frac{ \Psi(x,\delta')  - P(0)(x) + K' {\delta'}^{2 \gamma} }{\delta'^{2\gamma}}  \\
\nonumber & \leq &  \frac{ \Psi(x,t_0)  - P(0)(x) + K' t_0^{2 \gamma} }{\delta'^{2\gamma}} \\
 \nonumber &= & \exp( 2\gamma R(x))\left( \frac{ \Psi(x,t_0)  - P(0)(x) + K' t_0^{2 \gamma} }{\delta^{2\gamma}} \right) \\
& \leq &  R(x) \exp( 2 \gamma R(x)) =
 R(x)\exp( (\beta  + 1)  R(x)).
\end{eqnarray}
\smallskip

We conclude the proof by posing  
$c = \delta^{2 \gamma}, C_2 = K'= 2K + 2, C_1(x) =  R(x) \exp( (\beta+1) R(x)) - K', \delta(x) = \exp(-R(x)).$ 
Note that the function $R(x)$ goes to $+ \infty$ as $x $ approaches the boundary of $\Omega,$ while $\delta(x)$ goes to zero. 
As already mentioned, $\Omega=X$ if $\{\alpha\}$ is a K\"ahler class, hence in this case  we can find a uniform upper bound for the function $R(x)$ and a (uniform) lower bound for $\delta(x)$.
\end{proof}

\begin{lemma}\label{lemma_varcoeff}
There exists  continuous positive functions $K(x) > 2, $ and $\delta(x)$  in $\Omega $ of the form $K(x):= L R(x) e^{R(x)}$, and $\delta(x) := S e^{-R(x)}$ with $L>1$ and $0 < S< \frac{1}{2}$ constants, such that:
$$
\left( \alpha(z)+ i\partial \bar{\partial}_{(z,w)} \Psi(z,w) \right)[\theta,\tau]^2 \geq  
$$
\begin{equation}\label{varcoeff00} 
\left( -\left(A\lambda(z,|w|) +K(z)|w|^{2} \right) |dz|^2  -K(z)|w| |dw|^2\right)[\theta,\tau]^2 -K(z)|w||\theta||\tau|
\end{equation} \label{variable} 
if $0 < |w| < \delta(x).$
\end{lemma}

\begin{remark}
Since for every point $x_0 \in \Omega,$ we work in a neighborhood of $x_0$ where we choose local holomorphic  coordinates which are $\omega$ normal at $x_0,$
if we wish to state Lemma \ref{lemma_varcoeff} in a coordinate-free way we just have to replace  $|\theta|^2 =  |dz|^2[\theta, \tau]^2$ with $\omega[\theta]^2$ and   $|\theta|$ with $\sqrt{\omega[\theta]^2}.$  \label{normalcoord} 
\end{remark} 

\begin{proof}

We choose as above canonical local holomorphic coordinates $(z,\zeta)$ on the tangent bundle of $X$ centered at some point $(x_0,0),$ where $z$ varies in an open set $U,$ and the coordinates  are $\omega$ normal at $x_0.$ Note that for any $z\in U$, ${\exph}_{z}(0) = z$ and $(d\,{\exph}_{z})_0 = Id.$
Then, since $X$ is compact, there exists a positive constant  $S$ such that for any $|\zeta| < |w| < S$  and for any $z\in U$ the norm of determinant of the real Jacobian of ${\exph}_z(\zeta)$ as a function of $\zeta$  is between $1/2$ and $2$ and the function ${\exph}_z$ as a function of $\zeta$ is Lipschitz with Lipschitz constant $2$.  Let $B(z,r) := \{ z' :   |z'-z| < r \}.$
Let $\eta$ be a local potential of $\alpha,$ then
$P(0) + \eta$ is psh. Now fix $z_0 \in U.$ Recall that for any $j,k=1,\cdots n$ we have 

$$  \frac{1} {|w|^{2n}}\left| \int_{B(0,|w|)} \frac{\partial^2  (P(0) + \eta)}{ \partial z_j \partial \bar{z_k}}({\exph}_{z_0}(\zeta)) d \lambda(\zeta) \right|  \leq   \frac{C_3}{|w|^{2n}}\int_{B(0,|w|)} \Delta(P(0) + \eta) ({\exph}_{z_0}(\zeta)) d \lambda(\zeta),$$ 
where $d \lambda  $ is the Lebesgue measure and  $C_3$ is a positive constant. Here $\Delta := \sum_j\frac{ \partial^2  }{ \partial z_j \partial \bar{z_j}}$.

On the other hand, by the uniform bounds on the determinant of the real Jacobian and on the Lipschitz constant on ${\exph}_{z_0}$ we have:
\begin{flalign*}
& \frac{C_3}{|w|^{2n}}\int_{B(0,|w|)}   \Delta(P(0) + \eta) ({\exph}_{z_{0}}(\zeta)) d \lambda(\zeta) \\
  &\leq   \frac{C'_3}{|w|^{2n}}\int_{B(0,|w|)} \Delta(P(0) + \eta) ({\exph}_{z_0}(\zeta))\,|\det({\Jac}_{\mathbb{R}}\, {\exph}_{z_0})| \,{d} \lambda(\zeta)\\   
 &= \frac{C'_3}{|w|^{2n}}\int_{{{\exph}_{z_0}}(B(0,|w|))} \Delta(P(0) + \eta) (u) d \lambda(u ).
\end{flalign*}
Also,
$$|{\exph}_{z_0}(\zeta) -{\exph}_{z_0}( 0)|=|{\exph}_{z_0}(\zeta) - z_{0}| \leq 2|\zeta| \leq 2|w| \qquad {\rm on}\; |\zeta| < |w|.$$ This means that ${\exph}_{z_0}(B(0,|w|) ) \subseteq B(z_0,2|w|)$. 
 Therefore
 
\[  \frac{C'_3}{|w|^{2n}}\int_{{\exph}_{z_0}(B(0,|w|))} \Delta(P(0) + \eta)(u) \,d \lambda(u ) \leq   \frac{C'_3}{|w|^{2n}}\int_{ B(z_0,2|w|) } \Delta(P(0) + \eta) (u) \,d \lambda(u ).\]  
\smallskip

\noindent Given  $0 < a < 1,  t  > 0,  \gamma \in [0,1], $ by \eqref{eq00} and by \cite[eq. (1.16)]{BD}  we have:
\begin{eqnarray*}   
\frac{ \Psi(z_0,t)  - P(0)(z_0) }{t^{2 \gamma}} & = & \frac{ \Psi(P(0) + \eta)(z_0,t)  - P(0)(z_0) - \eta(z_0)}{t^{2 \gamma}} - O(t^{2-2\gamma})\\
&\geq & \frac{ C_5(a)}{t^{2 \gamma}} \left(   \frac{1}{(at)^{2n-2}} \int_{B(z_0,at)} \Delta (P(0) + \eta)(u)  d \lambda(u) \right) - C_4 t^{2 - 2 \gamma}
 \end{eqnarray*}
for  positive constants $C_4, C_5(a).$ 

The previous inequality with $ \gamma = \frac{1}{2}, 0 < a < 1, t > 0,$ gives
$$\frac{ \Psi(z_0,t)  - P(0)(z_0) }{t} \geq  a^2 C_5(a) \left(   \frac{1}{(at)^{2n}} \int_{B(z_0,at)} \Delta(P(0) + \eta)(u) \, d \lambda(u) \right)t - C_4 t.$$
Since condition $I_0$ holds, thanks to Lemma \ref{beta} we can ensure that $II_{1/2}$ holds as well. Thus we have \eqref{ineq_final_gamma} with $t = \delta' < \delta(x)$ and $\gamma=1/2$:
$$0\leq   \frac{ \Psi(z_0, t)  - P(0)(z_0) + K' t }{t} \leq R(z_0) \exp( R(z_0)).$$
Combining the above inequalities we get
$$R(z_0) \exp( R(z_0)) - K' + C_4 t  \geq  a^2 C_5(a) \left(   \frac{1}{(at)^{2n}} \int_{B(z_0,at)} \Delta(P(0) + \eta)(u)\,  d \lambda(u) \right)t, $$
i.e.
$$0 \leq  \left(   \frac{1}{(at)^{2n}} \int_{B(z_0,at)} \Delta(P(0) + \eta) (u)d \lambda(u) \right)  \leq \frac{R(z_0) \exp( R(z_0)) - K' + (\frac{C_4}{a}) at }{(a C_5(a))\,at}.  $$ 
Taking $a = 1/2, at = 2|w| $, for any $j,k=1,\cdots, n$ we get:

$$  \frac{1} {|w|^{2n}}\left| \int_{B(0,|w|)} \frac{\partial^2  (P(0) + \eta)}{ \partial z_j \partial \bar{z}_k}({\exph}_{z_0}(\zeta))\, d \lambda(\zeta) \right| \leq  C_6 |w|^{-1}\left({R(z_0)\exp(R(z_0)) - K' + 4 C_4 |w|}\right). $$ 
Observe that all the above constants are independent on the choice of the point $z_0 \in U.$
Now  we consider the global function   $M(z) :=  B' R(z)\exp{R(z)}$, where $B'>1$ is a constant such that  
$$M(z)>C_6 ({R(z)\exp(R(z)) - K' + 4 C_4 }). $$
Since $|w|\leq 1$, it then follows that for any $j,k$
\begin{eqnarray*}
 \int_{B(0,1)} \frac{\partial^2 (P(0)+\eta) }{\partial z_j \partial \bar{z}_k}  ({\exph}_{z_0}(w\zeta) ) d\lambda(\zeta) &= &  \frac{1} {|w|^{2n}} \int_{B(0,|w|)} \frac{\partial^2  (P(0) + \eta)}{ \partial z_j \partial \bar{z}_k}({\exph}_{z_0}(\zeta)) d \lambda(\zeta) \\
&\leq & M(z_0) |w|^{-1}.
\end{eqnarray*}
The above inequality ensures that 
\begin{equation*}
 \int_{B(0,1)} \frac{\partial^2 (P(0)+\eta) }{\partial z_j \partial \bar{z}_k}  ({\exph}_{z_0}(w\zeta) ) \,d\lambda(\zeta) = O( M(z_0) |w|^{-1}).
\end{equation*}
Moreover, the fact that $ \int_{B(0,1)} \frac{\partial^2 \eta }{\partial z_j \partial \bar{z}_k}  ({\exph}_{z_0}(w\zeta) ) d\lambda(\zeta)  = O(1)$ 
gives for  small $|w|$
\begin{equation}\label{M}
 \int_{B(0,1)} \frac{\partial^2 P(0) }{\partial z_j \partial \bar{z}_k} ({\exph}_{z_0}(w\zeta) ) \,d\lambda(\zeta) = O( M(z_0) |w|^{-1}).
\end{equation}
Observe that, a priori,  by \cite[eq. (3.11)]{Dem94} we only have that the left hand side quantity is $O(|w|^{-2})$.

We now follow the arguments in the proof of \cite[Theorem 4.1]{Dem94}. 
By \cite[Proposition 3.8]{Dem94} we have that 
$$\partial \bar{\partial}_{(z,w)} \Psi(z,w)[\theta,\tau]^2\big|_{z=z_0}= \int_{ \zeta \in T_{z_0} X }  \partial \bar{\partial}_{(z,w)} P(0) (\tau' \wedge \bar{\tau'} +|w|^2 V)_{{\exph}_{z_0}{(w\zeta})}\;  \chi(|\zeta|^2)  \; dV_\omega +O(|w|^{N-1})$$ 
where $N$ is a positive integer and the vector fields $\tau'$ and $V$ are the ones defined at \cite[page 8]{Dem94}. We observe that such vector fields depends on $N$ (the multi-indexes $\alpha, \beta$  are indeed such that $2\leq |\alpha|, |\beta|\leq N$). In order to obtain the estimates we want, we need to apply \cite[Proposition 3.8]{Dem94} with $N=3,$ whereas Demailly applies it with $N = 2.$
On the other hand, by the observation in the last lines of the proof of \cite[Proposition 3.8]{Dem94} we know that when we increase $N$ by one unit (in this case from $2$ to $3$) the contribution of the additional terms (which corresponds to the contribution of all multi-indexes $\alpha,\beta$  with $|\alpha|$ and $|\beta|$  equal to $N+1=3$) in the expansion of $\partial \bar{\partial} \Psi(P(0))_{(z,w)}$  are of the form 

$$\left( \int_{B(0,1)}\frac{\partial^2  P(0) }{\partial z_j \partial \bar{z}_k}  ({\exph}_{z_0}(w\zeta) ) d\lambda(\zeta) \right) O(|w|^3).$$  By \eqref{M} these terms are of type $O(M(z_0)|w|^2).$\\
 \smallskip
 It then follows that  

 \begin{eqnarray}\label{iteration_step}
&& \partial \bar{\partial}_{(z,w)} \Psi(z,w)[\theta,\tau]^2 \big|_{z=z_0}\\
\nonumber &&= \int_{ \zeta \in T_{z_0} X }  \partial \bar{\partial}_{(z,w)} P(0) (\tau' \wedge \bar{\tau'} +|w|^2 V)_{{\exph}_{z_0}{(w\zeta})}\;  \chi(|\zeta|^2)  \; dV_\omega  +O(M(z_0)|w|^2) + O(|w|^2),
 \end{eqnarray}
 where the vector fields $\tau'$ and $V$ are the ones defined with $N=2$. \\
 In the proof of \cite[Theorem 4.1]{Dem94} the estimates for $\partial \bar{\partial}_{(z,w)} \Psi(z,w)[\theta,\tau]^2$ are given up to terms $O(|w|)$ (coming from \cite[Proposition 3.8]{Dem94} applied with $N=2$), and  up to a term of the form 
$$\left( \int_{B(0,1)}\frac{\partial^2  P(0) }{\partial z_j \partial \bar{z}_k}  ({\exph}_{z_0}(w\zeta) ) d\lambda(\zeta) \right) O(|w|^3).$$
Therefore applying \cite[Proposition 3.8]{Dem94} with $N=3$ and observing that in our case the latter term is  $O(M(z_0)|w|^2)$ thanks to \eqref{M}, we can use the arguments in the proof of \cite[Theorem 4.1]{Dem94} to estimate $\partial \bar{\partial}_{(z,w)} \Psi(z,w)[\theta,\tau]^2$ adding an error term of the form $O(M(z)|w|^2).$

Thus 
\cite[eq. (4.3)]{Dem94}  and the formula $\tau = \theta + \eta \zeta +O(|w
|)$ at the end of page 14 in
\cite{Dem94} give that
$$
\left( \alpha(z)+ i \partial \bar{\partial}_{(z,w)} \Psi(z,w)\right)[\theta,\tau]^2 \geq  
 $$
$$  -\left(A\lambda(z,|w|) + C|w|^2 \right) |dz|^2   - CM(z)|w|^2\,( |dz|^2 + |dw|^2))[\theta,\tau]^2  - K''(|\theta||\tau|+ |\tau|^2) $$
where  $C>0$  and $ |w| < \delta_0 \leq 1.$ Thanks to \cite[eq. (3.11)]{Dem94} and \eqref{M}, the coefficient $K''$ is of the form  $O(M(z))|w|$. Moreover, we observe that the coefficient of $|dz|^2$ is of that form because of \eqref{lelong_a} and we stress again that the term $CM(z)|w|^2\,( |dz|^2 + |dw|^2)$ in the lower bound comes from \eqref{iteration_step}.

Now, besides the dependence on $M(z_0),$ as we already saw in the proof of Lemma \ref{ineq00}, the various ``$O$-terms"  are uniformly bounded, thus $O(|w|^2)$ is bounded by $C|w|^2(|\theta|^2 + |\tau|^2)$ for some positive constant $C.$ 
Finally we collect together all terms multiplying $|\theta|^2,$ all terms multiplying $|\eta|^2,$ and all terms multiplying $|\theta||\tau|,$   The estimate in \eqref{varcoeff00} then holds with $0 < |w| < Se^{-R(z)}$ and $K(z) := {L}M(z)$ for some positive constant $L.$
\end{proof}

\subsection{The K\"ahler case}
In the K\"ahler case the lower bound \eqref{varcoeff00} holds with constant coefficients. This allows us to reprove regularity in the K\"ahler case in the spirit of \cite{BD}. 

\begin{corollary}\label{kaheler}
If $\{ \alpha \}$ is a K\"ahler class, then the estimates  $I_1$ and $II_1 $   with constant coefficients holds on $X$. In particular the function $P(0)$ is $C^{1,\bar{1}}$ on $X$.  
\end{corollary}

\begin{proof}
By Lemma \ref{firstineq} we have $I_0.$  Since the class $\{ \alpha \}$ is K\"ahler, Lemma \ref{beta} ensures that $II_{1/2}$ holds with 
constant coefficients (i.e. with $C_1(x), \delta(x)$  constants). Using {the arguments in} Lemma \ref{variable} we derive \eqref{varcoeff00} where, since $\{\alpha\}$ is K\"ahler,  the function $ R(x)$  is uniformly bounded from  above by a big positive constant, hence inequality $I_1$ holds with constant coefficients. Then by Lemma \ref{beta} we have $II_1.$ By Lemma \ref{ineqreg} we conclude that $P(0)$ is $C^{1,\bar{1}}.$
\end{proof}
\smallskip

The corollary above was also proved in \cite{Ber19} using Monge-Amp\`ere equations and the so-called $\beta$-convergence.

\section{Lower bound of the hessian II: the trick of the line bundle}\label{section_line}
The same arguments in Corollary \ref{kaheler} can not be used in the big case since the coefficients appearing in the lower bound of the hessian are functions that explodes when approaching the boundary of $\Omega$. The idea is then to work on a high power of a suitable  line bundle in order to   ``kill" such functions after a  change of coordinates.
After that, the reasoning will be similar to the  above.

\medskip

 As mentioned in Section \ref{s1} , the K\"ahler current $T=\alpha+i \partial \bar{\partial} \Phi$ can be chosen to have analytic singularities along $\Sigma,$  so there exists a modification $q : \tilde{X} \rightarrow X$, (which is a sequence of blow-ups with smooth center)  such that $q^*T = \tilde{\omega} + \sum_{1 \leq j \leq N} c_j [D_j],$ where $c_j$ are positive real numbers, $\tilde{\omega}$ is a semipositive smooth real closed $(1,1)$-form and $[D_j]$ is the current of integration along an irreducible divisor $D_j$ \cite[Theorem 3.4]{DP}. Moreover the map $q$ gives a biholomorphism between $q^{-1}(\Omega)$ and $\Omega.$ \\
 Observe that if $\varphi$ is a non-positive $\alpha $-psh function of $X,$ then $\varphi \circ q$ is a non-positive $q^*\alpha$-psh function on $\tilde{X}.$ Vice-versa if $\psi$ is non-positive $(q^*\alpha)$-psh on $\tilde{X},$ then the restriction of $\psi$ to $q^{-1}(\Omega)$ is of the form $\varphi \circ q $ for some non-positive $\alpha $-psh function $\varphi$  on $\Omega.$ Such function $\varphi$ can be extended to all of $X,$ since $\Sigma=X\setminus \Omega$ is a closed pluripolar subset, hence $\psi = \varphi \circ q$ in all of $X.$ It follows that the largest non-positive $(q^*\alpha)$-psh function on $\tilde{X}$ is  $P(0) \circ q.$ Since $q$ gives a biholomorphism from $q^{-1}(\Omega)$ to $\Omega,$ showing that $P(0)$ is $C^{1,\bar{1}}$ in $\Omega$ is equivalent to show that so is $P(0) \circ q $ in 
$q^{-1}(\Omega).$ 
In the sequel we will then assume $X = \tilde{X},$  $ q = Id_X,$ $\Sigma = \bigcup_{1 \leq j \leq N}  D_j,$ and $T = \alpha + i \partial \bar{\partial} \Phi = \tilde{\omega} + \sum_{1 \leq j\leq N} c_j [D_j]$. 

Also, recall that if $h$ is an hermitian metric on an holomorphic line bundle $L,$  and $\sigma$ is a non-zero holomorphic section of $L,$ then the Poincar\'e-Lelong formula gives $[D] = \frac{i}{2\pi}\left(  \Theta_h+ \partial\bar{\partial} \log|\sigma|^2_h  \right),$ where  $\Theta_h$ s the curvature of $h$ and  $D$ is the zero set of $\sigma.$  
In particular if $L_j,$ $\sigma_j$ are, respectively,  the line bundle and the holomorphic section associated to the divisor $D_j$ and $h_j$ is a smooth hermitian  metric on $L_j$, we have
\[ T = \tilde{\omega} + \frac{i}{2 \pi}   \sum_{j=1}^N c_j   \left(  \Theta_{h_j} + \partial \bar{\partial} \log|\sigma_j|^2_{h_j} \right) =  \alpha + i \partial \bar{\partial}\Phi.  \]
Then $i \partial \bar{\partial} \left( \sum_j \frac{c_j}{ 2\pi}  \log|\sigma_j|^2_{h_j}  - \Phi\right)$ is smooth, hence 
\[ \Phi = \sum_j \frac{c_j}{2\pi}  \log|\sigma_j|^2_{h_j} + \mbox{a smooth function on $X$}. \]  

\noindent Given $p \in \Sigma$, let $U_p(1)$ and $U_p(2), U_p(3)$  open neighbourhoods of $p$ such that $U_p(2)$ is relatively compact in 
$U_p(1),$ and $U_p(3)$ is relatively compact in $U_p(2).$  
Assume that there exists in $U_p(1)$ holomorphic local coordinates $z$  for $X$ and holomorphic nowhere zero sections  of $L_j.$\\
Let $\{ {V'_i}\}_{\{ 1 \leq j \leq \ell \}}$ be a finite sub-covering of the covering 
$ \{U_p(2)\}_{p \in \Sigma }$ of $\Sigma,$ and set  $\{ {V_i}\}_{\{ 1 \leq j \leq \ell \}}$ the corresponding sub-covering of 
$\{ U_p(1) \}_{ p \in \Sigma }.$ \\
Let $U = \bigcup_{ p \in \Sigma} U_p(3),$ and $\tilde{H} = X \setminus U,$ then $\tilde{H}$ is compact.  Given $q \in \tilde{H}$, let $U_q(1)$ and $U_q(2)$  open neighborhoods of $q$ such that $U_q(2)$ is relatively compact in 
$U_q(1),$ $\overline{U_q(1)} \cap \Sigma = \emptyset$ and assume that in $U_q(1)$ there exists holomorphic local coordinates $z$ and holomorphic nowhere zero sections  of $L_j.$ Let $\{{V'_i}\}_{ \ell +1 \leq j \leq r}$ be a finite sub covering of the covering $\{ U_q(2) \}_{  q \in \tilde{H} }$ of $\tilde{H},$ and set  $\{{V_i}\}_{\{ \ell + 1 \leq j \leq r \}}$ the corresponding subcovering of 
$\{ U_q(1) \}_{q \in \tilde{H}}.$ 
\smallskip

Then we obtain  $\{ V_i \}_{\{1 \leq i \leq r \}}$  and  $\{ V'_i \}_{\{1 \leq i \leq r \}}$ finite open coverings of $X,$ with the following pro\-per\-ties: 
\begin{itemize}
\item[1.] $V'_i$ is relatively compact in $V_i,$  for $1 \leq i \leq r,$

\item[2.]   $\{V'_i \}_{ \{ 1 \leq i \leq \ell \}} $ covers $\Sigma,$
\item[3.]   $\overline{V_i} \cap \Sigma = \emptyset$ for $ \ell +1  \leq i \leq r$, 
\item[4.]   there exist local holomorphic coordinates $z$ on the  open  set $V_i,$ and nowhere zero holomorphic sections $\eta_{i,j}$ of the line bundle $L_j|_{V_i}$.
\end{itemize}
 Therefore we have 
 \begin{equation} 0 < a_{i,j} < |\eta_{i,j}|^2_{h_i} < b_{i,j}. \label{ab} \end{equation} on $V'_i$ for $1 \leq i \leq r,$ 
 for some positive constants $a_{i,j},b_{i,j}.$  Let $g_{i,j} := \frac{\sigma_i^2}{\eta_{i,j}^2}.$ By rescaling the section $\sigma_j$ (if necessary) we can assume that $ |\sigma_j|^2_{h_j} <\frac{ 1}{2}(\min_{i}  a_{i,j}).$ Recall that the covering   $\{V'_i \}_{ \{ 1 \leq i \leq r \}} $ is finite.  Then by construction it follows that $|g_{i,j}| < 1/2 $ on $V'_i.$ Moreover $g_{i,j} $ is holomorphic, it is  nowhere zero  on $V'_i \setminus D_j$  and  $g_{i,j} \equiv 0$ on $V'_i \cap D_j,$ for $1 \leq i \leq r$ and $1 \leq j \leq N.$ 
 \medskip

Also, by definition  on $V'_i,$    we can re-write $\Phi $  as
\begin{equation} \Phi = \sum_j \frac{c_j}{2\pi}  \log| g_{i,j}| + \mbox{a smooth  bounded function on $V'_i$}.  \label{analitsing} \end{equation}  
for $1 \leq i \leq r.$

Given a positive integer $m$ to be suitably chosen later, we let 
\[ L = (L_1 \otimes L_2 \ldots \otimes L_N)^{2m}, \] with its holomorphic section $s := (\sigma_1 \otimes \sigma_2  \ldots \otimes \sigma_N)^{2m}.$ Let $L^*$ the dual bundle of $L.$ Let $h = (h_1 \otimes h_2 \ldots \otimes h_N)^{2m}$ be the corresponding hermitian metric on $L,$ and let $h^*$ be the dual metric on $L^*.$  Let $M(L^*)$ be the total space of the line bundle $L^*$ and $\pi: M(L^*) \to X$ be the natural projection.  
Set
\[ W :=  \{ \beta   \in M(L^*) \setminus \{ \mbox{zero section} \} :    0 < | \beta(s(\pi(\beta))| < Se^{-R(\pi(\beta))}, \;\pi(\beta) \in \Omega \}, \] 
where $S$ is the constant appearing in Lemma \ref{lemma_varcoeff}.
Let $\tilde{\Psi} : W \to \R$ give by 
\[ \tilde{\Psi}(\beta) =: \Psi(\pi(\beta), \beta(s(\pi(\beta))).\] The map  $\tilde{\Psi}$ is smooth and for all $\beta$ in  
$M(L^*) \setminus \{ \mbox{zero section} \}$ we have  $ \lim_{\xi \in \C, \,\xi \to 0} \tilde{\Psi}( \xi \beta) = P(0)(\pi(\beta)).$

\noindent On each open set $V'_i,  1 \leq i \leq r $ we have a natural trivialization of $L^*$ given by $(z,\xi) \to (z,\xi \eta^*(z)),$ where $\eta^*$ is the dual section of 
$ \eta := (\eta_{i,1}  \otimes \eta_{i,2} \otimes \ldots \otimes \eta_{i,N})^{2m}.$ Then in local holomorphic coordinates we have 
\[ \tilde{\Psi}(z,\xi) = \Psi(z, G_{i}(z) \xi) \quad  {\rm on} \;V'_i \] where $G_{i} =  (g_{i,1}\, g_{i,2} \ldots g_{i,N})^m$.

\medskip

\noindent In the following we aim at computing the complex hessian of $\tilde{\Psi}$ in these coordinates.\\
Recall that the map $(z,\xi) \to (z,G_i(z) \xi)$ is holomorphic, hence if $(\theta,\tau)$ is a tangent vector to a point in $V'_i$  with $\theta$ tangent to $X,$ and $\tau \in \C,$  
 then the Levi form of $\tilde{\Psi}$ computed at the point $(z,\xi)$ on the vector $(\theta,\tau)$ is the  Levi form of $\Psi$ computed at the point $(z,G_i(z) \xi)$ on the vector 
 \begin{equation} {\Jac}( z,G_{i}(z) \xi)(\theta,\tau) = \left( \begin{array}{l}  \theta \\ \xi d_zG_{i}(\theta) + G_i(z) \tau  \end{array} \right) \label{jac} \end{equation}  
 where $\Jac$ is the holomorphic Jacobian.
\smallskip
 
\noindent Moreover, since for $1 \leq i \leq r$ the open set $V_i'$ is relatively compact in $V_i$ and both $\omega$ and $|dz|^2$ are K\"ahler forms in $V_i$, then $ \frac{ |dz|^2}{M}\leq \omega \leq  M|dz|^2$ on each $V'_i$ for some real number $M > 1.$ Since the covering is finite we may choose $M$ independent of $i.$\\
We then obtain the lower bound for the hessian of $\Psi$ in the coordinates on each $V'_i.$ More precisely, Lemma \ref{lemma_varcoeff}, Remark \ref{normalcoord} and the above observation insures that, up to replacing $A$ with $AM,$ and $K(z)$ with $K(z)M$ (but keeping the same name $A,K(z)$) we have
 
$$
\left( \alpha(z)+ \partial \bar{\partial}_{(z,w)} \Psi(z,w) \right)[\theta,\tau]^2 \geq  
$$
\begin{equation}\label{varcoeff} 
\left( -\left(A\lambda(z,|w|) +K(z)|w|^{2} \right) |dz|^2  -K(z)|w| |dw|^2\right)[\theta,\tau]^2 -K(z)|w||\theta||\tau|
\end{equation} 
where $K(z)={L} R(z)e^{R(z)}$, $L>1$, and $0 < |w| < Se^{-R(z)}$.
\medskip

If $\beta \in W,$ let $(z,\xi)$  are the coordinates of $\beta$ in the open set $V'_i,$ then from \eqref{varcoeff} and \eqref{jac},  we derive that 
\begin{flalign*}
&\left( \alpha(z)+ i \partial \bar{\partial}_{(z,\xi)} \tilde{\Psi}(z,\xi) \right)[\theta,\tau]^2  \\
 &\geq    -\left(A\lambda(z,|G_i(z)||\xi|) +K(z)|G_i(z)|^2 |\xi|^{2} \right) |\theta|^2  -K(z)|\xi||G_i(z)| |\xi d_zG_i(\theta)+ G_i(z) \tau)|^2\nonumber  \\
 & \quad\, - K(z)|\xi||G_i(z)| |\theta|| \xi d_zG_i(\theta) + G(z) \tau|   \nonumber \\
  & \geq    -\left(A\lambda(z,|G_i(z)||\xi|) +K(z)|G_i(z)|^2 |\xi|^{2} \right) |\theta|^2 
   -K(z)|\xi| |G_i(z)| (  |\xi |^2| {d_zG_i}(\theta)|^2+ |G_i(z)|^2 |\tau|^2 \nonumber \\
   & \quad\, + 2 |\xi||G_i(z)||{d_zG_i}(\theta)||\tau| ) -K(z)|\xi||G_i(z)|  |\theta|| \xi| |d_z{G_i}(\theta) |   -K(z)|\xi||G_i(z)||\theta| | G_i(z)|| \tau|   \nonumber\\
& \geq    -\left(A\lambda(z,|G_i(z)||\xi|) +K_1(z,\xi) \right)  |\xi|^{2}  |\theta|^2  -K_2(z)|\xi| |\tau|^2 - K_3(z,\xi) |\xi||\theta||\tau|  \nonumber\\
& \geq  -\left(A\lambda(z,|G_i(z)||\xi|) +K_1(z) \right)  |\xi|^{2}  |\theta|^2  -K_2(z)|\xi| |\tau|^2 - K_3(z) |\xi||\theta||\tau|  \nonumber. 
\end{flalign*}

\noindent Here 
\begin{eqnarray*}
 K_1(z,\xi) &:=& K(z)( |G_i(z)|^2 + |\xi||G_i(z)||d_z{G_i}|^2  + |G_i| |d_z{G_i}|) \\
 &\leq &  K(z)( |G_i(z)|^2 + |d_z{G_i}|^2  + |G_i| |d_z G_i ) := K_1(z) 
 \end{eqnarray*}
 where the last inequality follows since $|G_i(z)||\xi| \leq 1 $ and  $|d_z{G_i}(\theta)| \leq |d_z{G_i}||\theta|$;
$$ K_2(z) := K(z)|G_i(z)|^3 $$ and \[  K_3(z,\xi) :=   K(z)|G_i(z)| ( |G_i(z)| + 2|\xi||G_i(z)| |d_z{G_i}|) \leq K(z)|G_i(z)| ( |G_i(z)| + 2 |d_z{G_i}|):= K_3(z)   \]
where the inequality follows again from the fact that $|G_i(z)||\xi| \leq 1. $ Thus
 $$ \left( \alpha(z)+i \partial \bar{\partial}_{(z,\xi)} \tilde{\Psi}(z,\xi) \right)[\theta,\tau]^2 \geq  $$
\[  -\left(A\lambda(z,|G_i(z)||\xi|) +K_1(z) |\xi|^{2} \right) |\theta|^2  -K_2(z)|\xi| |\tau|^2 - K_3(z) |\xi||\theta||\tau|. \]

\medskip

\noindent We claim that we can  choose the positive integer $m$ so large that $K(z)|G_{i}(z)|  $ and $K(z) |d_z{G_i}|$ go to zero as $z$ approaches the set $\Sigma$. This would imply that $K_1(z),K_2(z),K_3(z)$ will also go to zero as $z$ approaches  $\Sigma.$\\
Now recall that $K(z) =L (Re^R)(z) \leq Le^{2R}(z)$, $R=B|\Phi|$ ($B>1$) and by \eqref{analitsing} $R \leq T_1 \sum_j \log(1/|g_{i,j}|)$. This means that 
$K(z) \leq \frac{T_2}{(g_{i,1}g_{i,2}\ldots g_{i,N})^{T_3}},$ with $T_1,T_2,T_3$ positive constants.  Then choosing $m > T_3 + 1 $ we have the claim.\\
On the other hand the functions  $K(z)G_i(z),$ and $K(z)|d_z{G_i}|$ are bounded also on $V'_i$ for $l+1 \leq i \leq r$ since  $\overline{V'_i} \cap \Sigma = \emptyset$ for $i \geq l+1$. We can then infer that 
$K_1(z) \leq H, K_2(z) \leq H|G_i(z)|^2$ and $ K_3(z) \leq H|G_i(z)|$ on each $V'_i$ for $1 \leq i \leq r$ with $H$ positive constant. 
Combining all the above we get:
\begin{eqnarray}\label{H} 
&& i \partial \bar{\partial}\tilde{\Psi}(z, \xi) [\theta,\tau]^2 \\
&& \geq   - \alpha[\theta]^2   -  \left(A \lambda(z, |G_{i}(z) \xi|) + H |\xi|^2 \right) |\theta|^2  - H|G_i(z)|^2 |\xi| |\tau|^2 - H|G_i| |\xi| |\theta| |\tau| \nonumber  \\
&&\geq   - \alpha[\theta]^2   -  \left(A \lambda(z, |G_{i}(z)\xi|) + \left(H + \frac{1}{2}\right)  |\xi|^2 \right) |\theta|^2  - \left(H |G_i(z)|^2 |\xi| +\frac{1}{2} H^2|G_i(z)|^2 \right) |\tau|^2\nonumber .
\end{eqnarray}
 where the last inequality follows from the fact that 
 $H |G_i| |\xi| |\theta| |\tau| \leq \frac{1}{2} \left( |\xi|^2 |\theta|^2 + H^2 |G_i|^2 |\tau|^2 \right)$.

\medskip

\noindent We also  claim  that there exists a positive constant $k_0 > 1 $ such that for $1 \leq i \leq \ell$, the quantity  $ \ \frac{e^{-k_0 R(z)}}{|G_i(z)|}$ goes to zero when $z$ goes to $\Sigma.$ \\
To prove the claim we observe that on $V'_i$  for $1 \leq i \leq \ell$,  we have: 

\[ -R = B \Phi \leq \frac{B}{2\pi} \left(\min_{1 \leq  j \leq N} c_j \right) \sum_j \log |g_{i,j}| + T_4 := T_5 \sum_j \log|g_{i,j}| + T_4 \]

with $T_4, T_5$ positive constants, since by construction $|g_{i,j}| < 1/2$ for $1 \leq i \leq \ell $ and $1 \leq j \leq N.$ It follows that $$\frac{e^{-k_0R}}{|G_i|} \leq   e^{k_0 T_4} |g_{i,1} g_{i,2} \cdots g_{i,N}|^{k_0 T_5 -m}.$$ It is then sufficient to choose $k_0 > \frac{m}{T_5} +1.$ Hence, we infer $\frac{e^{-k_0 R}}{|G_i|} < T_6$ for all $ 1 \leq i \leq r,$ with $T_6$ positive constant.\\
We fix $\delta$ real number  with $0 < \delta < S$ and we let 
\[ W_{\delta} := \{ \beta \in M(L^*) \setminus \{\mbox{zero section} \} :\; 0<|\beta(s(\pi(\beta))| < e^{-k_0R(\pi(\beta))}  \delta, \;  \pi(\beta) \in \Omega  \}. \]

\noindent Note that since $k_0 >1$, $W_{\delta} \subseteq W.$ Fix  $\beta\in W_\delta $, it follows that  if $(z,\xi)$ are the local coordinates of $\beta$ in $V'_i$ we have:
 $|\xi| <  T_6 \delta$.
 
\noindent We consider the function $\tilde{f} : W_{\delta}  \to \R$ given by 
$\tilde{f}(\beta)  := \tilde{\Psi}(\beta) + \tilde{K}' |\beta(s(\pi(\beta))|^2,$ with $\tilde{K'}$ a (big enough) positive constant. 
 In local coordinates:
\[ \tilde{f}(z,\xi) = \tilde{\Psi}(z,\xi) + \tilde{K}' |G_{i}(z)|^2 |\xi|^2 . \]
Using \eqref{jac}, we compute the complex Hessian of $\tilde{f}$ (w.r.t. $(z, \xi)$):

\begin{equation}  i\partial \bar{\partial}\tilde{f} [\theta,\tau]^2 = i \partial \bar{\partial}\tilde{\Psi} [\theta,\tau]^2 
+ \tilde{K}'  |\xi|^2 |d_zG_{i}(\theta)|^2 + \tilde{K'}|G_{i}(z)|^2|\tau|^2  + 2\tilde{K'} \Re(\overline{G_{i}(z)} d_zG_{i}(\theta) \xi \bar{\tau}).\label{Hessefdelta} \end{equation}
Then from \eqref{H} and \eqref{Hessefdelta}:
\begin{eqnarray*}
 i \partial \bar{\partial}\tilde{f} [\theta,\tau]^2 & \geq &  - \alpha[\theta]^2  -  \left(A \lambda(z, |G_{i}(z) \xi|) + (H + 1/2)  |\xi|^2 \right) |\theta|^2  - (H |G_i(z)|^2 |\xi| +1/2 H^2|G_i(z)|^2 ) |\tau|^2\\
 && + \tilde{K}'  |\xi|^2 |d_zG_{i}(\theta)|^2  + \tilde{K'}|G_{i}(z)|^2|\tau|^2  + 2\tilde{K'}\Re(\overline{G_{i}(z)} d_zG_{i}(\theta) \xi \bar{\tau}).
\end{eqnarray*}
On the other hand  since $|d_zG_i(\theta)|\leq |d_z G_i||\theta| $,
\begin{eqnarray*}
2\tilde{K}' |\Re(d_zG_{i}(\theta)\overline{G_{i}(z)} \xi \bar{\tau}) | &\leq& 2\tilde{K}'|G_{i}(z)||d_zG_i||\theta||\xi||\tau| \\
 &\leq & \tilde{K}'^2  |d_zG_{i}|^2|\xi|^2|\theta|^2 + |G_{i}|^2 |\tau|^2.
\end{eqnarray*}
Since for $\beta \in W_{\delta} $ we have $ |\xi| < T_6 \delta< T_6$, it follows that:  

\begin{eqnarray*}\label{fdelta3}
i \partial \bar{\partial}\tilde{f} [\theta,\tau]^2 & \geq &  - \alpha[\theta]^2  -  \left( A \lambda(z, |G_{i}(z) \xi|)  
 + (H + 1/2 +  (\tilde{K'}^2 + K')|d_zG_{i}|^2 ) |\xi|^2  \right) |\theta|^2  \\
 && + \left( \tilde{K}' |G_{i}(z)|^2 -|G_i(z)|^2 -  (H |G_i(z)|^2 |\xi| +1/2 H^2|G_i(z)|^2 ) \right) |\tau|^2  \\
&  \geq & - \alpha[\theta]^2 -  \left( A \lambda(z, |G_{i}(z) \xi|)  
 + (H + 1/2 + (\tilde{K'}^2 + K') |d_z G_{i}|^2 ) |\xi|^2  \right) |\theta|^2 \\
  && + \left( \tilde{K}' |G_{i}(z)|^2 -|G_i(z)|^2 -  (H T_6  |G_i(z)|^2  +1/2 H^2|G_i(z)|^2 ) \right) |\tau|^2 \\
  &= &  - \alpha[\theta]^2  -  \left( A \lambda(z, |G_{i}(z) \xi|)  
 + (H + 1/2 + (\tilde{K'}^2 + K')|d_z G_{i}|^2 ) |\xi|^2  \right) |\theta|^2 \\
&& + |G_i(z)|^2 \left( \tilde{K}'  - 1  -  H T_6    - 1/2 H^2  \right) |\tau|^2. \end{eqnarray*}
We choose $\tilde{K}'$ such that   $ \left( \tilde{K}'  - 1  -  H T_6    - \frac{1}{2} H^2  \right) > 0$  and such that $\tilde{K'} - K > 0$ (where $K$ is the constant appearing in \eqref{Lelong}). We also choose $K^*$ such that 
  $  K^* >H + \frac{1}{2} +  ( \tilde{K'}^2 + \tilde{K}')|d_zG_{i}|^2  $   for all $1 \leq i \leq r.$ \\
Finally we find:
 \begin{equation} i \partial \bar{\partial}\tilde{f} [\theta,\tau]^2 >  - \alpha[\theta]^2 -  \left( A \lambda(z, |G_{i}(z) \xi|)  
 + K^*  |\xi|^2  \right) |\theta|^2. \label{fdelta4} \end{equation}
Also we recall that  $\omega$ and $|dz|^2$  are K\"ahler forms on $V_i,$ hence $\frac{\omega}{ |dz|^2}$ is bounded on $V_i'.$ So by multiplying $A$ and $K^*$ by the same big positive  constant (but keeping the same name $A$ and $K^*$) we obtain:  
 \begin{equation} i \partial \bar{\partial}\tilde{f} [\theta,\tau]^2 >  - \alpha[\theta]^2  -  \left( A \lambda(z, |G_{i}(z) \xi|)  
 + K^*  |\xi|^2  \right)  \omega[\theta]^2. \label{fdelta5} \end{equation}

\smallskip

\noindent  We now  fix $x_0 \in \Omega$ and we look at the restriction of $\tilde{f}$ to the fiber  over $x_0$ of $\pi: W_{\delta}  \to \Omega$ as a function of $ u := \log|\beta(s(x_0))|.$  
Since $W_{\delta} \subseteq W,$ such restriction  $h_{x_0} $ is well defined on  $ ]- \infty, -k_0R(x_0) + \log\delta[$. Then $h_{x_0}$ is 
strictly increasing with respect to $u,$ since thanks to \eqref{Lelong}, the first derivative of $h_{x_0}$ is $\lambda(x_0,e^u )  + 2 (\tilde{K'} - K ))e^{2u}. $    
Observe that by construction the function $\tilde{f} : W_{\delta} \to \R $ is invariant under the natural $S^1$ action on $M(L^*) \setminus \{ \mbox{zero section} \}. $ \\
Now fix $c >0$.
Since the Lelong number of $P(0)$ at $x_0$ is zero we have that

$$ \lim_{u \to -\infty}  \frac{\partial (h_{x_0}(u) - cu)}{\partial u}  = -c.$$ 
This means that the function $h_{x_0}(u) - cu$ is either decreasing on the whole interval  $]-\infty, - k_0 R(x_0) + \log\delta[,$ and in this case $u_0 = -k_0R(x_0) + \log(\delta)$; or  $h_{x_0}(u) -cu$  has a minimum value $u_0<-k_0R(x_0)  +  \log\delta.$ In both cases $\lambda(x_0,e^{u_0}) \leq  c$. (Note that $u_0$ depends on $x_0$).
\smallskip

\noindent However, since $\tilde{\Psi} \leq 0$ and  $e^{-k_0R(x_0)}\delta < \delta$, by definition of $\tilde{f}$ we have
\begin{eqnarray*}
h_{x_0}(u_0)  - cu_0 &\leq & h_{x_0}( -k_0R(x_0) + \log\delta)  -c( -k_0R(x_0) +\log\delta) \\
&  \leq & \tilde{K'}  e^{-2k_0R(x_0)}\delta^2  -c(  -k_0 R(x_0) + \log \delta)\\ &\leq &  
\tilde{K'} \delta^2 -c(  - k_0R(x_0) + \log \delta )  \\
&\leq & \tilde{K'} \delta^2  + ck_0R(x_0) - c\log\delta. 
\end{eqnarray*}
Then:
\begin{equation} h_{x_0}(u_0) - cu_0  - \tilde{K'} \delta^2 -c(k_0 R(x_0) - \log\delta)\leq 0. \label{minore} \end{equation}  

\noindent Since the function $\beta \to \beta(s(\pi(\beta)),$ is holomorphic and nowhere zero on $W_{\delta},$ by  formula \eqref{fdelta5} and the fact that $|\xi|\leq T_6 \delta$, we see that  in coordinates
\begin{eqnarray*}
 i \partial \bar{\partial} ( \tilde{f} -c\log|\beta(s(\pi(\beta))|) =  i \partial \bar{\partial}  \tilde{f}  & > &  - \alpha - \left( A\lambda(\pi(\beta), \beta(s(\pi(\beta)) + K^* |\xi|^2 \right) \omega\\
 &\geq & - \alpha - \left( A\lambda(\pi(\beta), \beta(s(\pi(\beta)) + K^* T_6^2\delta^2 \right) \omega
 \end{eqnarray*}
Set $K^{**}= K^* T_6^2$ and let $\eta$,$\sigma$ be the local potentials of $\alpha $ and $\omega$ respectively. It then follows that the function  $\tilde{f} + (Ac  + K^{**}\delta^2)\sigma\circ \pi + \eta \circ \pi$ is strictly psh in a neighborhood of  a point $\beta_0 \in W_{\delta}$ such that 
$\pi(\beta_0) = x_0$ and $\log|\beta_0(s(x_0)| = u_0.$  

\smallskip

\noindent Since $\tilde{f}$ is invariant under rotations, we  can apply Kiselman minimum principle on an invariant neighbourhood of a point $\beta $ with $\pi(\beta) = x$   to prove that the function 
$ \tilde{\psi} : \Omega \to \R,$ given by 
$$\tilde{\psi}(x)  := \inf\left\{ \tilde{f}(\beta)  - c \log|\beta(s(x))|,\quad\, \beta \in W_{\delta}\,:\, \pi(\beta) = x \right\}$$ satisfies the inequality 
\begin{equation}\label{eqpsi}
 i \partial \bar{\partial} \tilde{\psi} \geq -C^*_0 \omega - \alpha
\end{equation}
  in a neighborhood of $x_0,$ with
\[ C^*_0 = Ac + K^{**} \delta^2. \] But $x_0$ is arbitrary so this inequality holds on all of $\Omega.$

\noindent Moreover by \eqref{minore} we find

\[   \tilde{\psi}(x)  - \tilde{K'} \delta^2 -c( k_0 R(x) - \log\delta) \leq 0 \]
We can rewrite this inequality as: 
\[ Q(x) :=  \tilde{\psi}(x) + ck_0 B \Phi(x)  + c \log\delta  - \tilde{K'} \delta^2 \leq 0.\]
Thanks to \eqref{eqpsi} and the fact that $i \partial \bar{\partial} \Phi \geq \omega - \alpha $ we obtain
\[ ( 1 + ck_0B) \alpha + i \partial \bar{\partial} Q \geq -  C^*_0 \omega  + ck_0B \omega \geq -C_0^* \omega. \]
Now we choose $\tilde{\mu} := \frac{C^*_0}{ C_0^* + 1 + c k_0B},$
and we note that
$0 < \tilde{\mu} < 1. $ Then the function $ v := (1 - \tilde{\mu}) Q+ \tilde{\mu} (1 + ck_0B) \Phi $ is $(1 + ck_0B) \alpha$ psh and $ v \leq 0.$

Then $v$ extends to a non-positive  $(1 + ck_0B) \alpha$ psh function on $X,$ so that

 \begin{equation} v \leq (1 + ck_0B)P(0) \leq P(0)  \label{bartheta} \end{equation} 
 since $P(0) \leq 0.$ In other words, since $Q \leq 0,$ 
 \[ P(0) \geq (1 - \tilde{\mu})Q + \tilde{\mu}(1 + ck_0B) \Phi \geq Q + \tilde{\mu}(1 +ck_0B) \Phi \] so 
 \begin{equation}  P(0) \geq \tilde{\psi}    + ck_0B\Phi +  \tilde{\mu}(1 + ck_0B)\Phi    + c\log \delta -\tilde{K'} \delta^2. \label{basicineq} \end{equation}  
 Now the function $\tilde{f}(\beta)$ with $\pi(\beta) = x,$ depends only on $\|\beta\|_{h^*}$ (where $h^*$ is the dual metric on $L^*$ introduced above). As function of $\|\beta\|_{h^*}$, $\tilde{f}$ is decreasing and it converges to $P(0)(x)$ 
when $\beta$ goes to zero along $\pi^{-1}(x).$
Then $P(0)(x) \leq \tilde{f}(\beta_0) = \tilde{\psi}(x) + c u_0,$ by definition of $u_0$ and $\tilde{\psi}$. 
Combining this with \eqref{basicineq} 
\[ P(0) - cu_0 \leq  \tilde{\psi} \leq P(0)    - ck_0B\Phi -  \tilde{\mu}(1 + ck_0B)\Phi   - c\log\delta  + \tilde{K'} \delta^2. \] 
Now $\Phi \leq -1$ and $P(0)$ is finite in $\Omega,$  then

\[ u_0 \geq - \left(k_0B|\Phi|  + \frac{\bar{\mu}}{c} (1 + ck_0B)|\Phi|   - \log\delta + \tilde{K'} \frac{\delta^2}{c} \right).\]
If we set $c = \delta^2$ we derive  $\frac{\bar{\mu}}{c}  \leq A  + K^{**}.$ 
We also set  $t_0 = e^{u_0},$  then
 \begin{eqnarray*}
 t_0 &\geq & \exp\left(- \left(  (A +  K^{**})(1 + ck_0B)|\Phi| + k_0B|\Phi|  + \tilde{K'}  \right) \right) \delta \\
 &\geq & 
  \exp\left(- \left(  (A +  K^{**})(1 + k_0B)|\Phi| + k_0B|\Phi|  + \tilde{K'}  \right) \right) \delta  
\end{eqnarray*}
since $c = \delta^2 < 1.$ 

\noindent On the other hand $t_0 \leq  \exp{(-k_0R(x))} \delta  \leq \delta,$ hence if we  set 
\[ F(x) := \left(  (A +  K^{**})(1 + k_0B)|\Phi(x)| + k_0B|\Phi(x)|  + \tilde{K'}  \right) \] 
we have $F > 0$ and   $e^{-F(x)} \delta \leq t_0 \leq \delta.$ \\
Let $V$ be a relatively compact open subset of $\Omega,$ 
and assume
 $ 0 < F  \leq M$ on $V,$ then 
 \begin{equation} e^{-M} \delta \leq t_0 \leq \delta \label{log} \end{equation} 
 on $V.$

However  by  \eqref{basicineq}, for $x \in V$  we  have:
\begin{eqnarray*}
 0  &\geq & \tilde{\psi}(x) - P(0)(x)    + ck_0B\Phi(x) +  \tilde{\mu}(1 + ck_0B)\Phi(x) + c\log\delta -\tilde{K'} \delta^2 \\ 
& = &\tilde{f}(\beta_0)-c\log t_0  - P(0)(x)    + ck_0B\Phi(x) +  \tilde{\mu}(1 + ck_0T_1)\Phi(x) + c\log\delta -\tilde{K'} \delta^2 \\
&= & \Psi(x,t_0) + \tilde{K'} t_0^2 -P(0)(x) - c\log t_0 + ck_0B\Phi(x) +  \tilde{\mu}(1 + ck_0B)\Phi(x)   + c\log\delta -\tilde{K'} \delta^2  
\end{eqnarray*}
that is 
\[ \Psi(x,t_0) + \tilde{K'} t_0^2 -P(0)(x) \leq c\log t_0 +ck_0B|\Phi|(x) + \tilde{\mu}(1+ck_0B)|\Phi|(x)  - c\log\delta + \tilde{K'} \delta^2. \]
Then, since $c = \delta^2 < 1$ and $\frac{\bar{\mu}}{c} \leq A + K^{**}$ we obtain 
\begin{eqnarray*} \frac{\Psi(x,t_0) + \tilde{K'} t_0^2 -P(0)(x)}{\delta^2}  &\leq  &\log \frac{t_0}{\delta} + k_0B|\Phi|(x) +  (A + K^{**})(1+ k_0B)|\Phi|(x)  + \tilde{K'}\\
&\leq &  k_0B|\Phi|(x) +  (A + K^{**})(1+ k_0B)|\Phi|(x)   + \tilde{K'}\\
&=& F(x).  
 \end{eqnarray*}
Hence
\begin{equation} \frac{\Psi(x,t_0) + \tilde{K'} t_0^2 -P(0)(x)}{\delta^2}  \leq  M \label{final1} \end{equation}
on the relatively compact open subset  $V$ on $\Omega$ chosen  above. 

\noindent Now recall that the function $ t \to \Psi(x,t) + \tilde{K'} t^2$ is increasing and it is greater or equal to $P(0),$ 
 then if we let $\delta' = e^{-M} \delta \leq t_0 $ we get
\[ 0 \leq  \frac{ \Psi(x,\delta') + \tilde{K'}\delta'^2  - P(0)(x)}{\delta^2} \leq M \]
 or equivalently, 
\[ 0 \leq  \frac{ \Psi(x,\delta') + \tilde{K'} \delta'^2  - P(0)(x)}{\delta'^2} \leq Me^{2M} \]
for $x \in V$ and $0 < \delta' =e^{-M} \delta< Se^{-M}$ (recall that $\delta<S<1/2$).\\
We then infer that for $x \in V $ and for $0 < \delta' < Se^{-M}$ 
\[ - \tilde{K'} \leq  \frac{ \Psi(x,\delta')   - P(0)(x)}{\delta'^2} \leq Me^{2M} - \tilde{K'} \]
Since the relatively compact set $V$ in $\Omega$ is arbitrary,
combining the above inequality together with Lemma \ref{ineqreg} we show that $P(0)$ in $C^{1,\bar{1}}$ in $\Omega.$
\smallskip
\begin{remark}
It is worth observing that the estimate \eqref{eqpsi} is essentially the one given in \cite[Lemma 1.12]{BD}.
\end{remark}
\bigskip

Given $k$ (possibly different) $C^{1,\bar{1}}$ functions $f_1, \ldots,f_k$ on $X,$ we can consider the rooftop envelope $P(\min_{1 \leq j  \leq k} f_j),$ also denoted by $P(f_1,\ldots,f_k)$.
Note that $P(f_1,\ldots,f_k)$ is a genuine $\alpha$-psh function with minimal singularities. In fact, the functions $f_j$ for $1 \leq j \leq k$ are continuous on $X,$ then there is a positive constant $C$ such that $-C \leq \min_{1 \leq j \leq k} f_j \leq C,$
therefore 
$ P(0) - C \leq P(f_1, \ldots,f_k)  \leq P(0) +  C.$

\smallskip

We are now able to derive our main theorem:

 \begin{theorem}\label{main}
Let $X$ be a compact complex manifold. Assume $\alpha$ is a smooth real closed $(1,1)$-form such that $ \{ \alpha \}$ is a big class. Let $f_1, \ldots, f_k$ be  $C^{1,\bar{1}}$ functions on $X$. Then the rooftop envelope $P(f_1,\ldots,f_k)$ is $C^{1,\bar{1}}$ in $\Omega$, the ample locus of $\{ \alpha \}.$ In particular the coefficients of 
$i \partial \bar{\partial}P(f_1,\ldots,f_k)$ are locally bounded on any local holomorphic coordinate system in $\Omega.$ 
\end{theorem}
\begin{proof}
We first assume that $(X,\omega)$ is K\"ahler,  that $k=1$ and that the function $f_1$ is quasi-psh.
As in \cite{BD}, we
observe that in the above computations we do not really need $\alpha$ to be a smooth form. In our arguments we only need that the coefficients of $\alpha$ are bounded on the open sets $V'_i, 1 \leq i \leq r$ given above. Observe that since $f_1$ is quasi-psh and it is
$ C^{1,\bar{1}}(X)$, then coefficients of $i \partial \bar{\partial} f_1$ are locally bounded. We can then replace the smooth form $\alpha$ by $\alpha + i \partial \bar{\partial}f_1$.\\
If $P(f_1)$  is the psh envelope of $f_1$ with respect to the form $\alpha,$ then $P(f_1) - f_1$ coincides with the envelope $P(0)$ with respect to
 $\alpha + i \partial \bar{\partial} f_1.$ 
 Moreover if $ q : \tilde{X} \to X$ is the modification as above, the function $f_1 \circ q$ is still $C^{1,\bar{1}}$ and quasi-psh on $\tilde{X}.$ Thus, we infer that $P(f_1)$ is $C^{1,\bar{1}}$ in $\Omega$. \\
We now take functions $f_1, \ldots f_k$ which are not necessarily quasi-psh and we choose $A>0$ big enough such that 
$A \omega > \alpha$. We then consider the rooftop envelope of $f_1, \ldots, f_k$ w.r.t. $A\omega$, $P_{A\omega}(f_1, \ldots,f_k)$. Since the functions $f_1, \ldots,f_k$ are in $C^{1,\bar{1}} (X)$ and $A\omega$ is a K\"ahler form it follows from \cite{Ber19} and \cite[Theorem 2.5]{DR_kis} that $P_{A\omega}(f_1, \ldots,f_k)$ is also $C^{1,\bar{1}}$ on $X$. 
 
Moreover $P_{A\omega}(f_1, \ldots,f_k)$  is also quasi-psh by definition, hence the arguments  above ensure that $P_\alpha(P_{A\omega}(f_1, \ldots,f_k))$ is $C^{1,\bar{1}}$ in $\Omega$.\\
We claim that $P_\alpha(f_1,\ldots,f_k)=P_\alpha(P_{A\omega}(f_1,\ldots,f_k))$. \\
Indeed, $P_\alpha( P_{A\omega}(f_1,\ldots,f_k))\leq P_\alpha( f_1,\ldots,f_k)$ since $P_{A\omega}(f_1,\ldots,f_k)\leq \min_{1 \leq j \leq k} f_j$. For the reverse inequality we observe that any $\alpha$-psh function is also $A \omega$-psh ($A\omega > \alpha$). This implies that $P_\alpha(f_1,\ldots,f_k)\leq P_{A\omega}(f_1,\ldots,f_k)$. Taking the envelope w.r.t $\alpha$ both sides we get the equality we wanted. 

Finally to remove the K\"ahler assumption: given a compact complex manifold $X$, we observe (as we did at the beginning of Section \ref{section_line}) that we can choose a K\"ahler current representing $\{\alpha\}$ which is smooth on $\Omega$. If we resolve its singularities and we make further blow-ups as in \cite[Theorem 3.4 and Lemma 3.5] {DP} we obtain a compact K\"ahler manifold $\tilde{X}$ with a birational morphism from  $\tilde{X}$ to $X$ which is biholomorphic over $\Omega.$ Then (as in Section \ref{section_line}) we can pull-back the  form $\alpha$ and the functions $P(0),f_1, \ldots,f_k$ in order to reduce to the K\"ahler case.    
\end{proof}

 \bibliographystyle{plain}
\bibliography{biblio.bib}
 \bigskip

 \noindent{\sc IMJ-PRG, Sorbonne Universit\'e, \\
 4 place Jussieu, 75005, Paris, France,\\
 \tt {eleonora.dinezza@imj-prg.fr}}
\bigskip
  
  \noindent{\sc Universit\`a di Roma TorVergata, \\
  Via della Ricerca Scientifica 1, 00133, Roma, Italy, \\
  \tt{trapani@mat.uniroma2.it}}

 \end{document}